\documentclass{amsart}
\usepackage{amssymb}
\usepackage[all]{xy}

\newtheorem{thm}{Theorem}[section]
\newtheorem{prop}[thm]{Proposition}
\newtheorem{cor}[thm]{Corollary}
\newtheorem{conj}[thm]{Conjecture}
\theoremstyle{remark}
\newtheorem{rmk}[thm]{Remark}
\newtheorem{exam}[thm]{Example}

\numberwithin{equation}{section}

\newcommand{\C}{\mathbb{C}}
\newcommand{\F}{\mathbb{F}}

\newcommand{\Oo}{\mathbb{O}}
\newcommand{\Z}{\mathbb{Z}}

\newcommand{\cF}{\mathcal{F}}
\newcommand{\cG}{\mathcal{G}}
\newcommand{\cH}{\mathcal{H}}
\newcommand{\cN}{\mathcal{N}}
\newcommand{\cO}{\mathcal{O}}
\newcommand{\cP}{\mathcal{P}}
\newcommand{\cQ}{\mathcal{Q}}
\newcommand{\cT}{\mathcal{T}}
\newcommand{\g}{\mathfrak{g}}
\newcommand{\fg}{\mathfrak{g}}
\newcommand{\fk}{\mathfrak{k}}
\newcommand{\fN}{\mathfrak{N}}
\newcommand{\dup}{{\mathrm{dup}}}
\newcommand{\SYT}{\mathrm{SYT}}

\DeclareMathOperator{\End}{End}

\DeclareMathOperator{\Ind}{Ind}
\DeclareMathOperator{\Irr}{Irr}
\DeclareMathOperator{\Des}{Des}
\DeclareMathOperator{\Lie}{Lie}
\DeclareMathOperator{\im}{im}

\title[The exotic Robinson--Schensted correspondence]
{The exotic Robinson--Schensted correspondence}
\author{Anthony Henderson}
\address{School of Mathematics and Statistics\\
University of Sydney NSW 2006\\
Australia}
\email{anthony.henderson@sydney.edu.au}
\author{Peter E. Trapa}
\address{Department of Mathematics\\
University of Utah\\
Salt Lake City, UT 84112-0090}
\email{ptrapa@math.utah.edu}


\thanks{This research was supported by
Australian Research Council grant DP0985184 and National Science Foundation grants DMS-0968275 and DMS-0968060.}

\begin{document}

\begin{abstract}
We study the action of the symplectic group on pairs of a vector and a flag. Considering the irreducible components of the conormal variety, we obtain an exotic analogue of the Robinson--Schensted correspondence. Conjecturally, the resulting cells are related to exotic character sheaves.
\end{abstract}

\maketitle

\section{Introduction}

The \emph{Robinson--Schensted correspondence} is a well-known bijection
\begin{equation} \label{rseqn}
S_n\;\longleftrightarrow\,\coprod_{\lambda\in\cP_n}\;\SYT(\lambda)\times\SYT(\lambda)\\
\end{equation}
where $S_n$ denotes the symmetric group, $\cP_n$ denotes the set of partitions of $n$, and $\SYT(\lambda)$ denotes the set of standard Young tableaux of shape $\lambda$. The resulting partition of $S_n$ into subsets indexed by $\cP_n$ is the partition into \emph{two-sided cells}, as defined by Kazhdan and Lusztig \cite[Section 5]{kl}. These two-sided cells occur in many representation-theoretic contexts: one which is particularly relevant here is the classification of unipotent character sheaves on $GL_n$ \cite[Section 18]{lusztig:charsheaves4}.

Steinberg \cite{st1,st2} gave a geometric interpretation of the Robinson--Schensted correspondence, by showing that both sides naturally parametrize the irreducible components of a certain variety $Z$. This variety can be defined for any connected reductive complex algebraic group $G$:
\[
Z=\{(gB,g'B,x)\in G/B\times G/B\times\cN\,|\,gB\in(G/B)_x, g'B\in(G/B)_x\},
\] 
where $G/B$ is the flag variety, $\cN$ is the nilpotent cone in $\Lie(G)$, and $(G/B)_x$ denotes the Springer fibre $\{gB\,|\,x\in\Lie(gBg^{-1})\}$. By considering the irreducible components of $Z$, Steinberg obtained in \cite{st1} a bijection
\begin{equation} \label{steinbergeqn}
W_G\;\longleftrightarrow\coprod_{x\in G\setminus\cN}A_x\setminus(\Irr(G/B)_x\times\Irr(G/B)_x)
\end{equation}
where $W_G$ is the Weyl group of $G$, $G\!\setminus\!\cN$ stands for a set of representatives of the orbits of $G$ in $\cN$, $A_x$ is the component group of the stabilizer of $x$ in $G$, and $\Irr(G/B)_x$ is the set of irreducible components of the Springer fibre. When $G=GL_n$, the nilpotent orbits are parametrized by $\cP_n$, irreducible components of $(G/B)_x$ are parametrized by standard Young tableaux, and $A_x$ is trivial for all $x$; in this case, Steinberg showed in \cite{st2} that \eqref{steinbergeqn} becomes the Robinson--Schensted correspondence \eqref{rseqn}.

Outside type $A$, the bijection \eqref{steinbergeqn} lacks some of the crucial combinatorial features of \eqref{rseqn}. In particular, the resulting partition of $W_G$ into subsets indexed by $G\!\setminus\!\cN$, known as \emph{geometric two-sided cells}, is not the partition into Kazhdan--Lusztig two-sided cells, and therefore is not directly relevant to the classification of character sheaves of $G$. See McGovern's paper \cite{mcgovern} for the geometric two-sided cells in the classical types.

Now let $V$ be a $2n$-dimensional complex vector space with a fixed nondegenerate skew-symmetric form, and write $G=GL(V)$, $K=Sp(V)$, $\g=\Lie(G)$, $\fk=\Lie(K)$. In \cite{kato:exotic,kato:deformations} Kato has shown that the representation of $K$ on $V\oplus\fg/\fk$ can be regarded as an `exotic' version of the adjoint representation of $K$ on $\fk$, and has somewhat neater combinatorics than the adjoint representation. For example, if $\fN$ denotes the \emph{exotic nilpotent cone} (the Hilbert nullcone of $V\oplus\fg/\fk$), then $K$ acts on $\fN$ with connected stabilizers and the orbits are in bijection with the irreducible representations of the Weyl group $W_K$ of $K$.

Kato's results suggest that there should be an interesting theory of \emph{exotic character sheaves}, which would be $K$-equivariant simple perverse sheaves on $V\times G/K$. The definition is easy to obtain by modifying Ginzburg's definition of character sheaves on $G/K$ \cite{ginzburg}, bearing in mind the parallel theory of \emph{mirabolic character sheaves} for $GL_n$ due to Finkelberg, Ginzburg, and Travkin \cite{fg1,fg2,fgt}. Here we will restrict attention to the \emph{unipotent} exotic character sheaves. These (or rather their pull-backs to $V\times G$) are the simple perverse constituents of the complexes $p_*q^!(\cF\boxtimes\cG)$, where the diagram
\begin{equation} \label{ginzburgdiagram}
\xymatrix{
V\times G/B\times G/B & V\times G\times G/B
\ar[l]_(0.45){q} \ar[r]^(0.6){p}
& V\times G
}
\end{equation}
is defined by $q(v,g,g'B)=(v,gg'B,g'B)$, $p(v,g,g'B)=(v,g)$, and $\cF$ and $\cG$
are $K$-equivariant simple perverse sheaves on $V\times G/B$ and $G/B$ respectively.
Here $q$ is smooth, $p$ is proper, and
both $p$ and $q$ are $(K\times K)$-equivariant, where 
\begin{itemize}
\item $K\times K$ acts on $V\times G$ by $(k_1,k_2)(v,g)=(k_1 v,k_1gk_2^{-1})$,
\item $K\times K$ acts on $V\times G\times G/B$ by $(k_1,k_2)(v,g,g'B)=(k_1 v,k_1gk_2^{-1},k_2g'B)$,
\item $K\times K$ acts on $V\times G/B\times G/B$ by $(k_1,k_2)(v,gB,g'B)=(k_1v,k_1gB,k_2g'B)$.
\end{itemize}

In order to study or classify such unipotent exotic character sheaves along the lines of Grojnowski's study of unipotent character sheaves \cite[Section 3]{grojnowski}, one would first need to find the cells for the action of $K\times K$ on $V\times G/B\times G/B$. Note that this is the direct product of separate actions of $K$ on $V\times G/B$ and on $G/B$.

The cells for the action of $K$ on $G/B$ are a special case of those defined for any symmetric space $G/K$ by Lusztig and Vogan \cite{lv}, using the action of the Iwahori--Hecke algebra of $W_G$ on the graded Grothendieck group of $K$-equivariant perverse sheaves on $G/B$. In our case, $K$ acts on $G/B$ with connected stabilizers, the orbits are parametrized by the subset $R_{2n}\subset S_{2n}$ consisting of those $w$ such that $ww_0$ is a fixed-point-free involution (where $w_0$ is the longest element of $S_{2n}$), and the cells were described combinatorially by Garfinkle \cite{garfinkle}. In fact, she gave a Robinson--Schensted-style bijection
\begin{equation} \label{garfinkleeqn}
R_{2n}\;\longleftrightarrow\,\coprod_{\lambda\in\cP_{2n}^\dup}\;\SYT(\lambda)
\end{equation}
where $\cP_{2n}^\dup$ consists of those partitions of $2n$ which are \emph{duplex}, meaning that all parts have even multiplicity. The resulting partition of $R_{2n}$ into subsets indexed by $\cP_{2n}^\dup$ is the partition into cells, and the bijection \eqref{garfinkleeqn} has a geometric interpretation similar to Steinberg's interpretation of \eqref{rseqn}: see \cite[Theorem 5.6]{trapa}.

What remains is to analyse the situation when $G/B$ is replaced by $V\times G/B$. The corresponding problem in the context of mirabolic character sheaves was addressed by Travkin \cite{travkin}, who studied the action of $GL_n$ on $\C^n\times GL_n/B_n\times GL_n/B_n$ (where $B_n$ denotes a Borel subgroup of $GL_n$), and the resulting \emph{mirabolic Robinson--Schensted correspondence}. Our work was inspired by his.

In Section 2, we re-formulate the parametrization of $K$-orbits in $V\times G/B$, due to Matsuki: they are in bijection with a modification $R_{2n}'$ of the set $R_{2n}$, defined in Proposition \ref{orbitprop}. In Section 3, we follow Steinberg's geometric approach to construct a bijection which we call the \emph{exotic Robinson--Schensted correspondence}:
\begin{equation}
R_{2n}'\;\longleftrightarrow\,\coprod_{(\mu;\nu)\in\cQ_{2n}'}\;\SYT(\mu+\nu),
\end{equation}
where $\cQ_{2n}'$ is a set of bipartitions, defined in \eqref{q2ndef}. The construction of this bijection involves both Kato's exotic nilpotent cone and the enhanced nilpotent cone of \cite{ah}. In Section 4, we state some conjectures relating this bijection to cells and exotic character sheaves.
\section{$K$-orbits in $V\times G/B$}
Fix a positive integer $n$. Let $V$ be a $2n$-dimensional complex vector space with a fixed nondegenerate skew-symmetric form $\langle\cdot,\cdot\rangle$. Write $G=GL(V)$, $K=Sp(V)$ as in the introduction. Let $B$ be a Borel subgroup of $G$. We will identify $G/B$ with the variety of complete flags $V_\bullet=(V_i)_{0\leq i\leq 2n}$ in $V$. Given a flag $V_\bullet$, we obtain another flag $V_{2n-\bullet}^\perp=(V_{2n-i}^\perp)$ by taking perpendicular subspaces for the form $\langle\cdot,\cdot\rangle$. 

It is well known that the $G$-orbits in $G/B\times G/B$ are in bijection with the symmetric group $S_{2n}$. For $w\in S_{2n}$, the corresponding orbit $\cO_w^{G}$ consists of pairs of flags $(U_\bullet,V_\bullet)$ such that
\begin{equation}
\dim U_i\cap V_j=|\{1,\cdots,i\}\cap w\{1,\cdots,j\}|,\text{ for all }i,j.
\end{equation}

It is also well known \cite[Proposition 10.4.1]{rs} that the $K$-orbits in $G/B$ are in bijection with the subset 
\[
R_{2n}=\{w\in S_{2n}\,|\,w(2n+1-i)=2n+1-w^{-1}(i)\neq i,\text{ for all }i\}.
\]
For $w\in R_{2n}$, the orbit $\cO_w$ consists of flags $V_\bullet$ such that $(V_\bullet,V_{2n-\bullet}^\perp)\in\cO_w^G$, or in other words
\begin{equation}
\dim V_i\cap V_{2n-j}^\perp=|\{1,\cdots,i\}\cap w\{1,\cdots,j\}|,\text{ for all }i,j.
\end{equation}

\begin{rmk} 
If $w_0$ denotes the longest element of $S_{2n}$ for the usual length function, defined by $w_0(i)=2n+1-i$, then $w\mapsto ww_0$ gives a bijection between $R_{2n}$ and the set of fixed-point-free involutions in $S_{2n}$. The latter set is also commonly used to parametrize the $K$-orbits in $G/B$, for instance in \cite[Proposition 6.1]{trapa}.
\end{rmk}

As was shown by Travkin \cite[Lemma 2]{travkin}, the $G$-orbits in $V\times G/B\times G/B$ are in bijection with the set $S_{2n}'$ of pairs $(w,\alpha)$ where $w\in S_{2n}$ and $\alpha$ is a subset of $\{1,2,\cdots,2n\}$ such that $i<j$, $w^{-1}(i)<w^{-1}(j)$ and $j\in\alpha$ together imply $i\in\alpha$. (Here we have set $\alpha=w(\beta)$ where $(w,\beta)$ is Travkin's parameter.) For $(w,\alpha)\in S_{2n}'$, the orbit $\cO_{w,\alpha}^G$ consists of triples $(v,U_\bullet,V_\bullet)$ such that $(U_\bullet,V_\bullet)\in\cO_w^G$ and
\begin{equation}
v\in U_i+V_j \Longleftrightarrow \alpha\subseteq \{1,\cdots,i\}\cup w\{1,\cdots,j\},\text{ for all }i,j.
\end{equation} 

The analogous statement in our context was proved by Matsuki:
\begin{prop} \label{orbitprop} \cite[Theorem 1.15(i)]{matsuki}
The $K$-orbits in $V\times G/B$ are in bijection with the subset $R_{2n}'$ of $S_{2n}'$ consisting of $(w,\alpha)$ for which $w\in R_{2n}$. For $(w,\alpha)\in R_{2n}'$, the orbit $\cO_{w,\alpha}$ consists of pairs $(v,V_\bullet)$ such that $(v,V_\bullet,V_{2n-\bullet}^\perp)\in\cO_{w,\alpha}^G$.
\end{prop}
\noindent
Note that $\{0\}\times G/B$ is a $K$-stable subvariety of $V\times G/B$, and is precisely the union of the orbits $\cO_{w,\emptyset}$ for $w\in R_{2n}$. So the $K$-orbits in $G/B$ can be thought of as a special case of the $K$-orbits in $V\times G/B$.

\begin{rmk}
Matsuki's statement actually refers to the $K$-orbits in $(V\setminus\{0\})\times G/B$, or equivalently the $K^v$-orbits in $G/B$ where $K^v$ is the stabilizer in $K$ of a nonzero $v\in V$. To translate from his parameter $(I_{(A)},I_{(X)},I_{(Y)},(c_{i,j})_{i,j\in I_{(A)}})$ to ours, one should first allow the extra possibility $I_{(X)}=I_{(Y)}=\emptyset$ (corresponding to the $v=0$ case). Then define $w\in R_{2n}$ so that $ww_0$ acts on $I_{(X)}$ and $I_{(Y)}$ as the unique order-preserving bijection between those sets and on $I_{(A)}$ as the involution whose permutation matrix is $(c_{i,j})$, and set $\alpha=\{i\,|\,\exists i'\in I_{(X)}, i'\geq i, w^{-1}(i')\geq w^{-1}(i)\}$.
\end{rmk}

Generalizing a well-known result for the $K$-action on $G/B$, we have:
\begin{prop} \label{connectedprop}
The stabilizer in $K$ of any point of $V\times G/B$ is connected.
\end{prop}
\begin{proof}
Matsuki gives a formula \cite[Proposition 1.14(ii)]{matsuki} for the number of $\F_q$-points of such a stabilizer defined over a finite field $\F_q$. His proof of this formula also shows the connectedness of such a stabilizer over $\C$.
\end{proof}

Matsuki also gives a formula for the number of orbits:
\begin{prop} \label{numberprop} \cite[Theorem 1.15(ii)]{matsuki}
With notation as above,
\[
|R_{2n}'|=\sum_{j=0}^n \frac{(2n)!}{2^{n-j}(n-j)!(j!)^2}.
\]
\end{prop}

A further problem is to describe the closure order on the $K$-orbits in $V\times G/B$. In \cite{magyar}, Magyar gave a combinatorial definition of a partial order $\leq$ on $S_{2n}'$ such that for any $(w,\alpha),(y,\beta)\in S_{2n}'$,
\begin{equation} \label{magyareqn}
\cO_{y,\beta}^G\subseteq\overline{\cO_{w,\alpha}^G}\;\Longleftrightarrow\;(y,\beta)\leq(w,\alpha).
\end{equation}
(In fact, Magyar used a slightly different parameter set, and effectively excluded the pairs $(w,\emptyset)$ since he considered $G$-orbits in $\mathbb{P}(V)\times G/B\times G/B$.) It is natural to conjecture that for any $(w,\alpha),(y,\beta)\in R_{2n}'$,
\begin{equation} \label{poconjeqn}
\cO_{y,\beta}\subseteq\overline{\cO_{w,\alpha}}\;\Longleftrightarrow\;(y,\beta)\leq(w,\alpha).
\end{equation}
The $\Longrightarrow$ direction follows immediately from \eqref{magyareqn}. We will not need \eqref{poconjeqn} here.

\begin{exam} \label{n2exam}
Let $n=2$. The three elements of $R_4$, written in one-line notation, are $1234$, $2143$, and $3412$. We will denote an element $(w,\alpha)\in R_4'$ by putting bars over the one-line notation for $w$ to indicate which elements belong to $\alpha$. For example, $\overline{21}4\overline{3}$ denotes $(2143,\{1,2,3\})$, and the corresponding orbit is
\[ \cO_{\overline{21}4\overline{3}}=\{(v,V_\bullet)\in V\times G/B\,|\,V_2^\perp=V_2, V_1^\perp\neq V_3, v\in V_3\setminus V_2\}, \]
with closure
\[ \overline{\cO_{\overline{21}4\overline{3}}}=\{(v,V_\bullet)\in V\times G/B\,|\,V_2^\perp=V_2, v\in V_3\}. \]
The Hasse diagram for the partial order $\leq$ on $R_4'$ is shown in Table \ref{hasse}. 
It is easy to verify \eqref{poconjeqn} in this case. All the orbit closures here are smooth except for 
\[ \overline{\cO_{\overline{3}4\overline{1}2}}=\{(v,V_\bullet)\in V\times G/B\,|\,v\in V_1^\perp\cap V_3\}, \]
whose singular locus is 
\[ \overline{\cO_{\overline{1}234}}=\{(v,V_\bullet)\in V\times G/B\,|\,V_1^\perp=V_3, v\in V_1\}. \]
\end{exam}

\begin{table}
\[
\xymatrix@R=20pt@C=10pt{
&&&& \overline{3412} \ar@{-}[dll]\ar@{-}[d]\ar@{-}[dr] &\\
&& \overline{2143} \ar@{-}[dll]\ar@{-}[dl]\ar@{-}[d] && \overline{3}4\overline{12} \ar@{-}[dlll]\ar@{-}[dl]\ar@{-}[d] & \overline{341}2 \ar@{-}[dlll]\ar@{-}[dll]\ar@{-}[d] \\
\overline{1234} \ar@{-}[dr] & \overline{21}4\overline{3} \ar@{-}[d]\ar@{-}[drr] & \overline{214}3 \ar@{-}[dl]\ar@{-}[dr] & \overline{3}4\overline{1}2 \ar@{-}[dll]\ar@{-}[d]\ar@{-}[dr]\ar@{-}[drr] & 34\overline{12} \ar@{-}[dl]\ar@{-}[d] & \overline{34}12 \ar@{-}[dll]\ar@{-}[d] \\
& \overline{123}4 \ar@{-}[dr] && \overline{21}43 \ar@{-}[dl]\ar@{-}[d]\ar@{-}[dr] & 34\overline{1}2 \ar@{-}[dl]\ar@{-}[dr] & \overline{3}412 \ar@{-}[dl]\ar@{-}[d] \\
&& \overline{12}34 \ar@{-}[dr] & 2\overline{1}43 \ar@{-}[d]\ar@{-}[dr] & \overline{2}143 \ar@{-}[dl]\ar@{-}[d] & 3412 \ar@{-}[dl] \\
&&& \overline{1}234 \ar@{-}[dr] & 2143 \ar@{-}[d] &\\
&&&& 1234 &
}
\]
\caption{Hasse diagram of $R_4'$.}\label{hasse}
\end{table}
\section{The conormal variety}
Continue the notation of the previous section. Let $\g=\Lie(G)=\End(V)$, and let $\cN$ be the nilpotent cone in $\g$. Let $\fk=\Lie(K)=\{y\in\g\,|\,y=-y^\perp\}$, where $y^\perp$ denotes the adjoint of $y$ for the symplectic form, i.e.\ the endomorphism such that
\begin{equation}
\langle yv_1,v_2\rangle=\langle v_1,y^\perp v_2\rangle,\text{ for all }v_1,v_2\in V.
\end{equation}
Let $S=\{y\in\g\,|\,y=y^\perp\}$ be the $K$-stable complementary subspace to $\fk$ in $\g$.

It is well known that the cotangent bundle $T^*(G/B)$ may be identified with the variety 
\[ \{(V_\bullet,x)\in G/B\times\cN\,|\,x(V_i)\subseteq V_{i-1},\,1\leq i\leq 2n\}. \] 
The projection onto the first factor is the cotangent bundle projection, and the projection onto the second factor is the moment map for the action of $G$ on $G/B$, after we identify $\g^*$ with $\g$ using the trace form. Recall that the condition $x(V_i)\subseteq V_{i-1}$, for $1\leq i\leq 2n$, is equivalent to $V_\bullet\in(G/B)_x$.

We can identify $V$ with $V^*$ via the map which sends $u\in V$ to the linear function $\langle u,\cdot\rangle:V\to\C$. Hence the cotangent bundle $T^*(V\times G/B)$ may be identified with the variety of quadruples $(v,u,V_\bullet,x)$ where $v,u\in V$ and $V_\bullet\in(G/B)_x$ as above. For $v,u\in V$, we define $\tau_{v,u}\in\g$ by $\tau_{v,u}(v')=\langle u,v'\rangle v$. This has rank $1$ (unless $v=0$ or $u=0$, in which case $\tau_{v,u}=0$), and every endomorphism of $V$ of rank $1$ is of this form.

For $(w,\alpha)\in R_{2n}'$, let $Y_{w,\alpha}\subset T^*(V\times G/B)$ be the conormal bundle over the orbit $\cO_{w,\alpha}$, and let $Y=\bigcup_{(w,\alpha)\in R_{2n}'}Y_{w,\alpha}$ be the conormal variety. Since each $Y_{w,\alpha}$ is irreducible of dimension equal to $\dim (V\times G/B)=2n^2+n$, the irreducible components of $Y$ are the closures $\overline{Y_{w,\alpha}}$. Our aim in this section, following \cite{st1} and \cite{travkin}, is to find a different parametrization of these irreducible components, and thus construct a bijection between $R_{2n}'$ and the other parameter set.

\begin{prop} \label{yprop}
With the above description of $T^*(V\times G/B)$, we have
\[
Y=\{(v,u,V_\bullet,x)\in T^*(V\times G/B)\,|\,x^\perp-x=\tau_{v,u}+\tau_{u,v}\}.
\]
\end{prop}
\begin{proof}
If we identify $\fk^*$ with $\fk$ using the trace form, then the natural linear map $\g^*\to\fk^*$ becomes the projection $\fg\to\fk:y\mapsto \frac{1}{2}(y-y^\perp)$. So the moment map for the action of $K$ on $V\times G/B$ is the map $\mu:T^*(V\times G/B)\to\fk$ defined by
\begin{equation}
\mu(v,u,V_\bullet,x)=\frac{1}{2}((\tau_{v,u}+x)-(\tau_{v,u}+x)^\perp)=\frac{1}{2}(\tau_{v,u}+\tau_{u,v}+x-x^\perp).
\end{equation}
Since $Y=\mu^{-1}(0)$, the result follows.
\end{proof}

We let $\pi:Y\to Z$ denote the obvious projection where
\begin{equation}
Z=\{(v,u,x)\in V\times V\times\cN\,|\,x^\perp-x=\tau_{v,u}+\tau_{u,v}\}.
\end{equation}
Note that each fibre $\pi^{-1}(v,u,x)$ can be identified with the Springer fibre $(G/B)_x$.

\begin{rmk} 
If $Y_0$ denotes the conormal variety for the action of $K$ on $G/B$, then the statement analogous to Proposition \ref{yprop} is that $Y_0=\{(V_\bullet,x)\in T^*(G/B)\,|\,x\in S\}$, and the analogue of $\pi$ is the projection $\pi_0:Y_0\to\cN\cap S$, which is the restriction to $Y_0$ of the moment map for the action of $G$ on $G/B$. We have no such moment-map interpretation of $\pi$ or $Z$. Note that $(v,u,x)\in Z$ does not imply $x\in S$.
\end{rmk}

Here is a vital fact about the variety $Z$.
\begin{prop} \label{nilpprop}
If $(v,u,x)\in Z$, then 
\[ \langle x^i v, x^j v\rangle=\langle x^i u, x^j u\rangle=\langle x^i v, x^j u\rangle=0,\text{ for all $i,j\geq 0$.} \]
\end{prop}
\begin{proof}
The assumption $x^\perp-x=\tau_{v,u}+\tau_{u,v}$ means that for all $v_1,v_2\in V$,
\begin{equation} \label{formeqn}
\langle v_1,xv_2\rangle-\langle xv_1,v_2\rangle=\langle u,v_1\rangle\langle v,v_2\rangle+
\langle v,v_1\rangle\langle u,v_2\rangle.
\end{equation}
Setting $v_1=x^i v$, $v_2=x^j v$ in \eqref{formeqn} gives
\begin{equation}
\langle x^i v,x^{j+1} v\rangle-\langle x^{i+1}v,x^j v\rangle=\langle u,x^i v\rangle\langle v,x^j v\rangle+
\langle v,x^i v\rangle\langle u,x^j v\rangle.
\end{equation}
From this we can deduce by induction on $\max\{i,j\}$ that $\langle x^i v, x^j v\rangle=0$ for all $i,j\geq 0$ (noting that the $i=j$ case is trivially true). The proof that $\langle x^i u, x^j u\rangle=0$ is identical. Setting $v_1=x^i v$, $v_2=x^j u$ in \eqref{formeqn} then gives
\begin{equation} \label{2formeqn}
\langle x^i v,x^{j+1}u\rangle-\langle x^{i+1}v,x^j u\rangle=-\langle x^i v,u\rangle\langle v,x^j u\rangle.
\end{equation}
If we assume that $\langle x^i v,u\rangle =0$ for all $i\geq 0$, \eqref{2formeqn} allows us to deduce by induction on $j$ that $\langle x^i v, x^j u\rangle=0$ as required; similarly if we assume that $\langle v,x^j u\rangle=0$ for all $j\geq 0$. So it suffices to find a contradiction to the assumption that there is some $i_0$ such that $\langle x^{i_0} v,u\rangle \neq 0$, and some $j_0$ such that $\langle v,x^{j_0} u\rangle\neq 0$. Since $x$ is nilpotent, we can assume that $i_0$ and $j_0$ are maximal with these properties. Then by downward induction on $i$, we can prove using \eqref{2formeqn} that $\langle x^i v,x^j u\rangle=0$ whenever $i>i_0$; similarly, we can prove this whenever $j>j_0$. But then when we set $i=i_0$ and $j=j_0$ in \eqref{2formeqn}, the left-hand side is zero and the right-hand side is nonzero, giving the desired contradiction.
\end{proof}
\begin{cor} \label{nilpcor}
If $(v,u,x)\in Z$, then $x+\tau_{v,u}\in\cN\cap S$.
\end{cor}
\begin{proof}
The assumption $x^\perp-x=\tau_{v,u}+\tau_{u,v}$ is equivalent to $x+\tau_{v,u}\in S$. So we need only show that $x+\tau_{v,u}$ is nilpotent. Let $W=\C[x]v$ be the subspace of $V$ spanned by $\{v,xv,x^2v,\cdots\}$. Since $W$ is stable under $x$, and also contains the image of $\tau_{v,u}$, it is stable under $x+\tau_{v,u}$, and the endomorphisms of $V/W$ induced by $x$ and by $x+\tau_{v,u}$ are the same. Since $\langle x^i v,u\rangle=0$ for all $i$ by Proposition \ref{nilpprop}, $W$ is contained in the kernel of $\tau_{v,u}$, so the endomorphisms of $W$ induced by $x$ and by $x+\tau_{v,u}$ are the same. So $x+\tau_{v,u}$ induces nilpotent endomorphisms of $W$ and of $V/W$, and is therefore nilpotent.
\end{proof}

Recall from \cite[Section 2]{ah} and \cite[Section 2]{travkin} that the $G$-orbits in the \emph{enhanced nilpotent cone} $V\times\cN$ are in bijection with the set $\cQ_{2n}$ of bipartitions $(\mu;\nu)$ of $2n$. To say that $(\mu;\nu)$ is a \emph{bipartition} of $2n$ is to say that $\mu=(\mu_1,\mu_2,\mu_3,\cdots)$ and $\nu=(\nu_1,\nu_2,\nu_3,\cdots)$ are partitions (weakly decreasing sequences of nonnegative integers, eventually zero) such that $|\mu|+|\nu|=2n$, where $|\mu|=\mu_1+\mu_2+\mu_3+\cdots$. For $(\mu;\nu)\in\cQ_{2n}$, the corresponding orbit $\cO_{\mu;\nu}\subset V\times\cN$ consists of those pairs $(v,x)$ such that $x$ has Jordan type $\mu+\nu=(\mu_1+\nu_1,\mu_2+\nu_2,\cdots)\in\cP_{2n}$ and the endomorphism of $V/\C[x]v$ induced by $x$ has Jordan type $(\mu_2+\nu_1,\mu_3+\nu_2,\cdots)\in\cP_{2n-\mu_1}$ (which forces $\dim \C[x]v=\mu_1$). 

Similarly, by \cite{kato:exotic} and \cite[Section 6]{ah}, the $K$-orbits in the \emph{exotic nilpotent cone} $\fN=V\times(\cN\cap S)$ are in bijection with the set $\cQ_n$ of bipartitions of $n$. The orbit $\Oo_{\mu;\nu}\subset \fN$ is the intersection of $\fN$ with $\cO_{\mu\cup\mu;\nu\cup\nu}$, where $\mu\cup\mu$ denotes the duplex partition $(\mu_1,\mu_1,\mu_2,\mu_2,\cdots)$.

We define
\begin{equation} \label{q2ndef}
\cQ_{2n}'=\{(\mu;\nu)\in\cQ_{2n}\,|\,\mu_1-\mu_2+\mu_3-\mu_4+\cdots=\nu_1-\nu_2+\nu_3-\nu_4+\cdots\}.
\end{equation}
In terms of the diagrammatic representation of bipartitions used in \cite{ah}, this condition says that the number of odd-length columns to the left of the wall is the same as the number of odd-length columns to the right of the wall.
\begin{prop} \label{balanceprop}
If $(v,u,x)\in Z$, then $(v,x)\in\cO_{\mu;\nu}$ for some $(\mu;\nu)\in\cQ_{2n}'$. We then have $(v,x+\tau_{v,u})\in\Oo_{\widetilde{\mu};\widetilde{\nu}}$ where
\[
\begin{split}
\widetilde{\mu}_1&=\mu_1\\
\widetilde{\nu}_1&=\nu_1-(\mu_1-\mu_2)\\
\widetilde{\mu}_2&=(\mu_1-\mu_2+\mu_3)-(\nu_1-\nu_2)\\
\widetilde{\nu}_2&=(\nu_1-\nu_2+\nu_3)-(\mu_1-\mu_2+\mu_3-\mu_4)\\
\widetilde{\mu}_3&=(\mu_1-\mu_2+\mu_3-\mu_4+\mu_5)-(\nu_1-\nu_2+\nu_3-\nu_4)\\
\widetilde{\nu}_3&=(\nu_1-\nu_2+\nu_3-\nu_4+\nu_5)-(\mu_1-\mu_2+\mu_3-\mu_4+\mu_5-\mu_6)\\
&\ \vdots
\end{split}
\]
\end{prop}
\begin{proof}
Certainly we have $(v,x)\in\cO_{\mu;\nu}$ for some $(\mu;\nu)\in\cQ_{2n}$.
By Corollary \ref{nilpcor}, we know that $(v,x+\tau_{v,u})\in\Oo_{\widetilde{\mu};\widetilde{\nu}}$ for some $(\widetilde{\mu};\widetilde{\nu})\in\cQ_n$. By definition, this implies that as an element of $V\times\cN$, $(v,x+\tau_{v,u})$ belongs to the orbit $\cO_{\widetilde{\mu}\cup\widetilde{\mu};\widetilde{\nu}\cup\widetilde{\nu}}$. As seen in the proof of Corollary \ref{nilpcor}, $x+\tau_{v,u}$ acts on $\C[x]v$ and $V/\C[x]v$ in the same way as $x$ does, so we must have
\begin{equation} \label{balanceeqn}
\begin{split}
\widetilde{\mu}_1&=\mu_1\\
\widetilde{\mu}_1+\widetilde{\nu}_1&=\mu_2+\nu_1\\
\widetilde{\mu}_2+\widetilde{\nu}_1&=\mu_3+\nu_2\\
\widetilde{\mu}_2+\widetilde{\nu}_2&=\mu_4+\nu_3\\
\widetilde{\mu}_3+\widetilde{\nu}_2&=\mu_5+\nu_4\\
\widetilde{\mu}_3+\widetilde{\nu}_3&=\mu_6+\nu_5\\
&\ \vdots
\end{split}
\end{equation}
This obviously implies the formulas for $\widetilde{\mu}_i,\widetilde{\nu}_i$ given in the statement. Then from the fact that $\widetilde{\mu}_i=\widetilde{\nu}_i=0$ for sufficiently large $i$ we deduce that $(\mu;\nu)\in\cQ_{2n}'$.
\end{proof}
Note that if $v=0$, then $\mu=\widetilde{\mu}=\emptyset$ and $\nu$ is the duplex partition $\widetilde{\nu}\cup\widetilde{\nu}$.

Let $\Psi:\cQ_{2n}'\to\cQ_n:(\mu;\nu)\mapsto(\widetilde{\mu};\widetilde{\nu})$ be the map defined in Proposition \ref{balanceprop}. For fixed $(\widetilde{\mu};\widetilde{\nu})\in\cQ_n$, $\Psi^{-1}(\widetilde{\mu};\widetilde{\nu})$ is the subset of $\cQ_{2n}$ consisting of all $(\mu;\nu)$ satisfying \eqref{balanceeqn}. In the notation of \cite[Section 7]{ahj}, this set is
$\cQ_{\widetilde{\mu}_1,(\widetilde{\mu}_1+\widetilde{\nu}_1,\widetilde{\mu}_2+\widetilde{\nu}_1,\widetilde{\mu}_2+\widetilde{\nu}_2,
\widetilde{\mu}_3+\widetilde{\nu}_2,\cdots)}$.

Recall from \cite[Definition 3.6]{ah} the partial order $\leq$ on $\cQ_{2n}$ which corresponds to the closure order on orbits in the enhanced nilpotent cone, Namely, $(\rho;\sigma)\leq(\mu;\nu)$ if and only if the following
inequalities hold for all $k\geq 0$:
\begin{equation}
\begin{split}
\rho_1+\sigma_1+\rho_2+\sigma_2+\cdots+\rho_k+\sigma_k
&\leq \mu_1+\nu_1+\mu_2+\nu_2+\cdots+\mu_k+\nu_k,\text{ and}\\
\rho_1+\sigma_1+\cdots+\rho_k+\sigma_k+\rho_{k+1}
&\leq \mu_1+\nu_1+\cdots+\mu_k+\nu_k+\mu_{k+1}.
\end{split}
\end{equation}
As observed in \cite[Proposition 7.2(2)]{ahj}, the restriction of this partial order to $\Psi^{-1}(\widetilde{\mu};\widetilde{\nu})$ is given by
\begin{equation}
(\rho;\sigma)\leq(\mu;\nu)\Longleftrightarrow\sigma_i\leq\nu_i\text{ for all }i\geq 1.
\end{equation}
Note that $(\widetilde{\mu}\cup\widetilde{\mu};\widetilde{\nu}\cup\widetilde{\nu})$ belongs to $\Psi^{-1}(\widetilde{\mu};\widetilde{\nu})$, and is in general neither a minimal nor a maximal element for this partial order.
\begin{exam}
If $n=2$ and $\widetilde{\mu}=\widetilde{\nu}=(1)$, then $\Psi^{-1}(\widetilde{\mu};\widetilde{\nu})$ consists of the three bipartitions $(1^3;1),(1^2;1^2),(1;21)$, on which the partial order is a total order.
\end{exam}

For any $(\mu;\nu)\in\cQ_{2n}'$, we define
\begin{equation}
Z_{\mu;\nu}=\{(v,u,x)\in Z\,|\,(v,x)\in\cO_{\mu;\nu}\}.
\end{equation}
Thus $Z$ is the disjoint union of the locally closed subvarieties $Z_{\mu;\nu}$. It is not immediately clear that every $Z_{\mu;\nu}$ is nonempty; this will be shown in Corollary \ref{dimcor}. 

For any $(\widetilde{\mu};\widetilde{\nu})\in\cQ_n$, define
\begin{equation}
Z^{\widetilde{\mu};\widetilde{\nu}}=\{(v,u,x)\in Z\,|\,(v,x+\tau_{v,u})\in\Oo_{\widetilde{\mu};\widetilde{\nu}}\}.
\end{equation}
Then by Proposition \ref{balanceprop} we have
\begin{equation}
Z^{\widetilde{\mu};\widetilde{\nu}}=\bigcup_{(\mu;\nu)\in\Psi^{-1}(\widetilde{\mu};\widetilde{\nu})} Z_{\mu;\nu}.
\end{equation}
For any $(\mu;\nu)\in\Psi^{-1}(\widetilde{\mu};\widetilde{\nu})$, we define
\begin{equation}
Z^{\widetilde{\mu};\widetilde{\nu}}_{\leq(\mu;\nu)}=\bigcup_{\substack{(\rho;\sigma)\in\Psi^{-1}(\widetilde{\mu};\widetilde{\nu})\\(\rho;\sigma)\leq(\mu;\nu)}} Z_{\rho;\sigma},
\end{equation}
which is clearly a closed subvariety of $Z^{\widetilde{\mu};\widetilde{\nu}}$.
\begin{prop} \label{bundleprop}
For any $(\widetilde{\mu};\widetilde{\nu})\in\cQ_n$ and $(\mu;\nu)\in\Psi^{-1}(\widetilde{\mu};\widetilde{\nu})$, the map
\[
\psi:Z^{\widetilde{\mu};\widetilde{\nu}}_{\leq(\mu;\nu)}\to\Oo_{\widetilde{\mu};\widetilde{\nu}}:(v,u,x)\mapsto(v,x+\tau_{v,u})
\]
is a fibre bundle whose fibres are isomorphic to affine space of some dimension $d_{\mu;\nu}$ \textup{(}or else, for the moment, there is the possibility that $Z^{\widetilde{\mu};\widetilde{\nu}}_{\leq(\mu;\nu)}$ is empty\textup{)}. 
\end{prop}
\begin{proof}
Fix $(v,y)\in\Oo_{\widetilde{\mu};\widetilde{\nu}}$, and let $W=\C[y]v$ (whose dimension is $\widetilde{\mu}_1=\mu_1$). Then for $u\in V$, the condition that $y-\tau_{v,u}$ is nilpotent is equivalent to the condition that $u\in W^\perp$, by Proposition \ref{nilpprop} and the proof of Corollary \ref{nilpcor}. If this condition holds, then $\im(\tau_{v,u})\subseteq W\subseteq\ker(\tau_{v,u})$, so
\begin{equation}
(y-\tau_{v,u})^k=y^k-\sum_{j=1}^{k}y^{k-j}\tau_{v,u}y^{j-1},\text{ for all }k\geq 0.
\end{equation}
Now the fibre $\psi^{-1}(v,y)$ can be identified with $\{u\in W^\perp\,|\,(v,y-\tau_{v,u})\in\overline{\cO_{\mu;\nu}}\}$. It follows immediately from \cite[Proposition 7.2(3)]{ahj} that for $u\in W^\perp$, the condition $(v,y-\tau_{v,u})\in\overline{\cO_{\mu;\nu}}$ is equivalent to
\begin{equation}
(y-\tau_{v,u})^{\mu_1+\nu_i}((y^{\mu_{i+1}+\nu_i})^{-1}(W))=0,\text{ for all }i\geq 1.
\end{equation}
Since the matrix coefficients of $(y-\tau_{v,u})^k$ are affine-linear functions of the coefficients of $u$, these equations define an affine-linear subspace of $W^\perp$ (conceivably empty). As $(v,y)$ ranges over the $K$-orbit $\Oo_{\widetilde{\mu};\widetilde{\nu}}$, the fibres $\psi^{-1}(v,y)$ clearly fit together into a bundle as required.
\end{proof}
\begin{cor} \label{irredcor}
$Z_{\mu;\nu}$ is an irreducible variety of dimension
\[
2n^2+d_{\mu;\nu}-2\nu_2-2\mu_3-2\nu_3-2\mu_4-4\nu_4-4\mu_5-4\nu_5-4\mu_6-6\nu_6-\cdots
\]
\textup{(}or else, for the moment, there is the possibility that $Z_{\mu;\nu}$ is empty\textup{)}.
\end{cor}
\begin{proof}
Certainly the orbit $\Oo_{\widetilde{\mu};\widetilde{\nu}}$ is irreducible.
By \cite[Theorem 2.20]{ahs} and \eqref{balanceeqn},
\begin{equation}
\begin{split}
\dim\Oo_{\widetilde{\mu};\widetilde{\nu}}&=2n^2-2\widetilde{\nu}_1-4\widetilde{\mu}_2-6\widetilde{\nu}_2-8\widetilde{\mu}_3-10\widetilde{\nu}_3-\cdots\\
&=2n^2-2(\widetilde{\mu}_2+\widetilde{\nu}_1)-2(\widetilde{\mu}_2+\widetilde{\nu}_2)-4(\widetilde{\mu}_3+\widetilde{\nu}_2)-4(\widetilde{\mu}_3+\widetilde{\nu}_3)-\cdots\\
&=2n^2-2(\mu_3+\nu_2)-2(\mu_4+\nu_3)-4(\mu_5+\nu_4)-4(\mu_6+\nu_5)-\cdots\\
&=2n^2-2\nu_2-2\mu_3-2\nu_3-2\mu_4-4\nu_4-4\mu_5-4\nu_5-4\mu_6-\cdots.
\end{split}
\end{equation}
By Proposition \ref{bundleprop}, $Z^{\widetilde{\mu};\widetilde{\nu}}_{\leq(\mu;\nu)}$ is an irreducible variety of dimension equal to this plus $d_{\mu;\nu}$ (or is empty).
Since $Z_{\mu;\nu}$ is an open subvariety of $Z^{\widetilde{\mu};\widetilde{\nu}}_{\leq(\mu;\nu)}$, the result follows. 
\end{proof}

Now we return to the variety $Y$. Recall that for any $x\in\cN$, the Springer fibre $(G/B)_x$ is the variety of flags $V_\bullet\in G/B$ such that $x(V_i)\subseteq V_{i-1}$ for all $i$. To any such flag we can associate the sequence of partitions which give the Jordan types of $x|_{V_1},x|_{V_2},\cdots$. These can be regarded as a standard tableau of shape $\lambda$, where $\lambda\in\cP_{2n}$ is the Jordan type of $x$. In this way, $(G/B)_x$ is the disjoint union of locally closed subvarieties $(G/B)_x^T$ where $T$ runs over $\SYT(\lambda)$. As shown by Spaltenstein \cite{spaltenstein} and Steinberg \cite{st1}, each one of these subvarieties is irreducible of dimension
\begin{equation} \label{spaltdimeqn}
\lambda_2+2\lambda_3+3\lambda_4+4\lambda_5+\cdots,
\end{equation}
and their closures are exactly the irreducible components of $(G/B)_x$.

Let $\cT_{2n}'$ be the set of pairs $((\mu;\nu),T)$ where $(\mu;\nu)\in\cQ_{2n}'$ and $T\in\SYT(\mu+\nu)$. For any $((\mu;\nu),T)\in\cT_{2n}'$, we define
\begin{equation}
Y_{\mu;\nu}^T=\{(v,u,V_\bullet,x)\in Y\,|\,(v,x)\in\cO_{\mu;\nu},V_\bullet\in(G/B)_x^T\}.
\end{equation}
Thus $Y$ is the disjoint union of the locally closed subvarieties $Y_{\mu;\nu}^T$ (some of which, for the moment, are conceivably empty). 

The main result of this section is:
\begin{thm} \label{irredcompthm}
Each $Y_{\mu;\nu}^T$ is nonempty and irreducible of dimension $2n^2+n$. Hence the irreducible components of $Y$ are the closures $\overline{Y_{\mu;\nu}^T}$ as $((\mu;\nu),T)$ runs over $\cT_{2n}'$. 
\end{thm}
\begin{proof}
It is clear that $\pi:Y_{\mu;\nu}^T\to Z_{\mu;\nu}$ is a fibre bundle whose fibres are isomorphic to $(G/B)_x^T$ where $x\in\cN$ has Jordan type $\mu+\nu$. By Corollary \ref{irredcor} and \eqref{spaltdimeqn}, $Y_{\mu;\nu}^T$ is an irreducible variety of dimension
\begin{equation} \label{dimeqn}
2n^2+d_{\mu;\nu}+\mu_2-\nu_2+\mu_4-\nu_4+\mu_6-\nu_6+\cdots
\end{equation}
\textup{(}or else, for the moment, there is the possibility that $Y_{\mu;\nu}^T$ is empty\textup{)}. Hence we can define a map $f:\cT_{2n}'\to R_{2n}'$ such that $Y_{\mu;\nu}^T\subseteq\overline{Y_{f((\mu;\nu),T)}}$. Since $Y=\bigcup Y_{\mu;\nu}^T$, $f$ must be surjective. If we can show that $|\cT_{2n}'|=|R_{2n}'|$, then it will follow that $f$ is bijective, implying that $Y_{\mu;\nu}^T$ is dense in $\overline{Y_{f((\mu;\nu),T)}}$ for all $((\mu;\nu),T)\in\cT_{2n}'$, which gives the result. By Proposition \ref{numberprop}, we are reduced to showing that
\begin{equation}
|\cT_{2n}'|=\sum_{j=0}^n \frac{(2n)!}{2^{n-j}(n-j)!(j!)^2}.
\end{equation}
This follows by taking dimensions of both sides in Proposition \ref{repprop} below, and using the well-known fact that $\dim V_\lambda=|\SYT(\lambda)|$.
\end{proof}

For $\lambda\in\cP_{2n}$, let $V_{\lambda}$ denote the irreducible complex representation of $S_{2n}$ labelled by $\lambda$ (with the usual convention that $V_{(2n)}$ is the trivial representation and $V_{(1^{2n})}$ the sign representation).
\begin{prop} \label{repprop}
We have an isomorphism of representations of $S_{2n}$,
\[
\bigoplus_{(\mu;\nu)\in\cQ_{2n}'}V_{\mu+\nu}
\;\cong\; \bigoplus_{j=0}^n\,\Ind_{W(C_{n-j})\times S_j\times S_j}^{S_{2n}}(\C_\delta),
\]
where $W(C_{n-j})$ is the subgroup of $S_{2n-2j}$ centralizing some fixed-point-free involution, $\delta$ is the linear character of $W(C_{n-j})$ obtained by restricting the sign character of $S_{2n-2j}$, and $\C_{\delta}$ is a $1$-dimensional vector space on which $W(C_{n-j})$ acts via $\delta$ and $S_j\times S_j$ acts trivially.
\end{prop}
\begin{proof}
It is well known that
\begin{equation}
\Ind_{W(C_{n-j})}^{S_{2n-2j}}(\C_\delta)\cong \bigoplus_{\sigma\in\cP_{2n-2j}^\dup}V_{\sigma},
\end{equation}
where $\cP_{2n-2j}^\dup$ denotes the set of duplex partitions of $2n-2j$, i.e.\ those in which every part has even multiplicity. Using the Pieri rule, it follows that the multiplicity of $V_\lambda$ in $\Ind_{W(C_{n-j})\times S_j\times S_j}^{S_{2n}}(\C_\delta)$ is the number of ways to successively remove two horizontal strips of size $j$ from the diagram of $\lambda$ so that what remains has only even-length columns. 

Any bipartition $(\mu;\nu)\in\cQ_{2n}'$ such that $\mu+\nu=\lambda$ determines such a pair of horizontal strips, as follows. The columns of the diagram of $\lambda$ are the union of the columns of $\mu$ and the columns of $\nu$ -- for definiteness, say that among columns of the same length, those belonging to $\mu$ come after those belonging to $\nu$. The first horizontal strip consists of one box from each column of $\mu$. The second horizontal strip consists of one box from each odd-length column of $\nu$ and one box from each (originally) even-length column of $\mu$. By the assumptions on $(\mu;\nu)$, both strips have size $\mu_1$. It is easy to check that every choice of two horizontal strips as above arises in this way for a unique $(\mu;\nu)$, and the result follows.
\end{proof}

As a consequence of Theorem \ref{irredcompthm}, we have:
\begin{cor} \label{dimcor}
For every $(\mu;\nu)\in\cQ_{2n}'$, $Z_{\mu;\nu}$ is indeed nonempty, and the quantity $d_{\mu;\nu}$ in Proposition \ref{bundleprop} and Corollary \ref{irredcor} equals $|\nu|$.
\end{cor}
\begin{proof}
Since each $Y_{\mu;\nu}^T$ is nonempty, $Z_{\mu;\nu}=\pi(Y_{\mu;\nu}^T)$ is nonempty.
Comparing Theorem \ref{irredcompthm} with \eqref{dimeqn}, we find that
\begin{equation}
d_{\mu;\nu}=n-\mu_2+\nu_2-\mu_4+\nu_4-\mu_6+\nu_6-\cdots.
\end{equation}
Now since $(\mu;\nu)\in\cQ_{2n}'$, we have
\begin{equation}
\mu_1+\nu_2+\mu_3+\nu_4+\mu_5+\nu_6+\cdots=
\nu_1+\mu_2+\nu_3+\mu_4+\nu_5+\mu_6+\cdots=n,
\end{equation}
and it follows that $d_{\mu;\nu}=\nu_1+\nu_2+\nu_3+\cdots=|\nu|$.
\end{proof}

More importantly, we now have our exotic Robinson--Schensted correspondence:
\begin{cor} \label{bijectioncor}
There is a bijection
\[ R_{2n}'\;\longleftrightarrow\;\cT_{2n}' \]
such that $(w,\alpha)\in R_{2n}'$ corresponds to $((\mu;\nu),T)\in\cT_{2n}'$ if and only if $\overline{Y_{w,\alpha}}=\overline{Y_{\mu;\nu}^T}$. 
\end{cor}
\begin{proof}
This is the bijection $f$ mentioned in the proof of Theorem \ref{irredcompthm}. 
\end{proof}
An equivalent way to characterize this correspondence is: $(w,\alpha)\in R_{2n}'$ corresponds to $((\mu;\nu),T)\in\cT_{2n}'$ if and only if, for generic $(v,u,V_\bullet,x)\in Y_{w,\alpha}$, we have $(v,x)\in\cO_{\mu;\nu}$ and $V_\bullet\in(G/B)_x^T$.

We do not know a combinatorial description of this correspondence. The method used by Travkin \cite[Section 3.4]{travkin} to relate his mirabolic Robinson--Schensted correspondence to the ordinary one does not appear to work in our setting.

\begin{exam} \label{exoticrsexam}
Let $n=2$ as in Example \ref{n2exam}. We have
\[
\cQ_{4}'=\{(2;2),(21;1),(1;21),(2^2;\emptyset),(1^2;1^2),(\emptyset;2^2),(1^3;1),(1;1^3),(1^4;\emptyset),(\emptyset,1^4)\}.
\]
Easy calculations show that the exotic Robinson--Schensted correspondence in this case is as shown in Table \ref{exoticrstable}. For any $\lambda\in\cP_4$, it happens that a standard tableau $T\in\SYT(\lambda)$ is determined by its descent set $\Des(T)$, consisting of those $i\in\{1,2,3\}$ such that $i+1$ is in a lower row than $i$; for ease of reference in connection with Conjecture \ref{mainconj}, we have listed $\Des(T)$ rather than $T$ itself.
\end{exam}

\begin{table}
\[
\begin{array}{|c||c|c|}
\hline
(w,\alpha)&(\mu;\nu)&\Des(T)\\
\hline
1234&(\emptyset;2^2)&\{2\}\\
\overline{1}234&(1;21)&\{2\}\\
\overline{12}34&(2;2)&\emptyset\\
\overline{123}4&(21;1)&\{2\}\\
\overline{1234}&(2^2;\emptyset)&\{2\}\\
2143&(\emptyset;2^2)&\{1,3\}\\
2\overline{1}43&(1;21)&\{3\}\\
\overline{2}143&(1;21)&\{1\}\\
\overline{21}43&(1^2;1^2)&\{1,3\}\\
\overline{21}4\overline{3}&(21;1)&\{1\}\\
\overline{214}3&(21;1)&\{3\}\\
\overline{2143}&(2^2;\emptyset)&\{1,3\}\\
\hline
\end{array}
\qquad\qquad
\begin{array}{|c||c|c|}
\hline
(w,\alpha)&(\mu;\nu)&\Des(T)\\
\hline
3412&(\emptyset;1^4)&\{1,2,3\}\\
34\overline{1}2&(1;1^3)&\{2,3\}\\
\overline{3}412&(1^3;1)&\{1,2\}\\
\overline{3}4\overline{1}2&(1^2;1^2)&\{2\}\\
34\overline{12}&(1;1^3)&\{1,3\}\\
\overline{34}12&(1^3;1)&\{1,3\}\\
\overline{3}4\overline{12}&(1;1^3)&\{1,2\}\\
\overline{341}2&(1^3;1)&\{2,3\}\\
\overline{3412}&(1^4;\emptyset)&\{1,2,3\}\\
\hline
\end{array}
\]
\bigskip
\caption{Exotic Robinson--Schensted correspondence for $n=2$}\label{exoticrstable}
\end{table} 
\section{Conjectures}
Continue the notation of the previous sections. Let $\cP_K(V\times G/B)$ be the category of $K$-equivariant perverse sheaves on $V\times G/B$, and define 
\[ M_K(V\times G/B)=K_0(\cP_K(V\times G/B))\otimes_\Z \Z[v,v^{-1}]. \]
By Proposition \ref{connectedprop}, this is a free $\Z[v,v^{-1}]$-module with basis $\{C_{w,\alpha}\,|\,(w,\alpha)\in R_{2n}'\}$ where $C_{w,\alpha}$ is the image in $K_0(\cP_K(V\times G/B))$ of the intersection cohomology complex of $\overline{\cO_{w,\alpha}}$ (shifted so as to be perverse). Define $M_G(G/B\times G/B)$ similarly. 

There is a well-known convolution product on $M_G(G/B\times G/B)$ which makes it isomorphic to the Iwahori--Hecke algebra $\cH_{2n}$ of $S_{2n}$. A similar convolution construction makes $M_K(V\times G/B)$ into an $\cH_{2n}$-module. This is a direct generalization of the $\cH_{2n}$-module $M_K(G/B)$ defined by Lusztig and Vogan \cite{lv} (indeed, $M_K(G/B)$ occurs in $M_K(V\times G/B)$ as the $\Z[v,v^{-1}]$-submodule spanned by $C_{w,\emptyset}$ for $w\in R_{2n}$). By analogy with \cite{lv} and \cite{travkin}, we are led to the following conjecture. A reference for the $W$-graph terminology used here is \cite[Chapter 11]{gpf}.

\begin{conj} \label{mainconj}
\begin{enumerate}
\item The basis $\{C_{w,\alpha}\}$ of the $\cH_{2n}$-module $M_K(V\times G/B)$ is a $W$-graph basis. 
\item If $(w,\alpha)$ corresponds to $((\mu;\nu),T)$ under the exotic Robinson--Schensted correspondence, the label of $C_{w,\alpha}$ in the $W$-graph is the descent set $\Des(T)$.
\item If $(w,\alpha)$ corresponds to $((\mu;\nu),T)$ and $(y,\beta)$ corresponds to $((\rho;\sigma),U)$, then $(w,\alpha)\succeq(y,\beta)$ in the $W$-graph preorder if and only if $(\mu;\nu)\geq(\rho;\sigma)$ and $(\nu;\mu)\geq(\sigma;\rho)$. In particular, $(w,\alpha)$ and $(y,\beta)$ lie in the same cell if and only if $(\mu;\nu)=(\rho;\sigma)$. Hence the cells are indexed by $\cQ_{2n}'$.
\item For $(\mu;\nu)\in\cQ_{2n}'$, the corresponding cell module is the irreducible $\cH_{2n}$-module labelled by the partition $\mu+\nu$. Hence the representation of $S_{2n}$ on $M_K(V\times G/B)|_{v=1}$ is isomorphic to those mentioned in Proposition \ref{repprop}. 
\end{enumerate}
\end{conj}

It is straightforward to verify Conjecture \ref{mainconj} in the cases $n=1$ and $n=2$.

\begin{rmk}
It is known \cite{trapa} that the cells for the $(GL_n\times GL_n)$-action on $G/B$ are indexed by the sign tableaux of signature $(n,n)$. There is a simple bijection between these sign tableaux and $\cQ_{2n}'$: transpose rows to columns, and assign columns beginning with $+$ to the $\mu$ side of the bipartition, and those beginning with $-$ to the $\nu$ side. We do not know any deeper connection between the two types of cells.
\end{rmk}

We return now to the setting of the introduction, and the proposed classification of unipotent exotic character sheaves. Imitating the arguments of \cite[Theorem 3.4]{grojnowski}, one can assign to each unipotent exotic character sheaf a pair $(c_1,c_2)$ where $c_1$ is a cell for the action of $K$ on $V\times G/B$ and $c_2$ is a cell for the action of $K$ on $G/B$, with the property that the two-sided cell in $S_{2n}$ associated to $c_1$ is the same as that for $c_2$. If Conjecture \ref{mainconj} holds, then such pairs $(c_1,c_2)$ are in bijection with pairs $((\mu;\nu),\lambda)\in\cQ_{2n}'\times\cP_{2n}^\dup$ with the property that $\mu+\nu=\lambda$, or in other words with $(\mu;\nu)\in\cQ_{2n}$ such that $\mu+\nu$ is duplex. These in turn are in bijection with $\cQ_n$ via the map $\cQ_n\to\cQ_{2n}:(\mu;\nu)\mapsto(\mu\cup\mu;\nu\cup\nu)$, and of course $\cQ_n$ parametrizes the irreducible representations of $W_K=W(C_n)$ as in \cite[Section 5.5]{gpf}. Hence we would have a map from the set of isomorphism classes of unipotent exotic character sheaves to the set of isomorphism classes of irreducible representations of $W_K$. It seems plausible that this map is a bijection, as with ordinary character sheaves for $GL_n$ (but in contrast to the situation of ordinary character sheaves for $K$). If true, this would be another respect in which the exotic picture has neater combinatorics than the usual picture for the symplectic group.

\textbf{Acknowledgements.} We would like to thank Kyo Nishiyama for encouraging comments, and for inviting us to Tokyo for the September 2011 workshop in honour of the 60th birthday of Jiro Sekiguchi, where this work was completed.


\begin{thebibliography}{Trap}

\bibitem[AH]{ah}
P.~N.~Achar and A.~Henderson, {\it Orbit closures in the enhanced nilpotent cone},
Adv.\ Math.\ {\bf 219} (2008), no.~1, 27--62.

\bibitem[AHJ]{ahj}
P.~N.~Achar, A.~Henderson, and B.~F.~Jones, {\it Normality of orbit closures in the enhanced nilpotent cone},
Nagoya Math.\ J.\ {\bf 203} (2011), 1--45.

\bibitem[AHS]{ahs}
P.~N.~Achar, A.~Henderson, and E.~Sommers, {\it Pieces of nilpotent cones for classical groups}, 
Represent.\ Theory {\bf 15} (2011), 584--616.



\bibitem[FG1]{fg1}
M.~Finkelberg and V.~Ginzburg, {\it Cherednik algebras for algebraic curves}, in
{\it Representation Theory of Algebraic Groups and Quantum Groups}, 121--153,
Progr.\ Math.\ {\bf 284}, Birkh\"auser/Springer, New York, 2010.

\bibitem[FG2]{fg2}
M.~Finkelberg and V.~Ginzburg, {\it On mirabolic $\mathcal{D}$-modules},
Int.\ Math.\ Res. Not.\ IMRN {\bf 2010}, no.~15, 2947--2986.

\bibitem[FGT]{fgt}
M.~Finkelberg, V.~Ginzburg, and R.~Travkin, {\it Mirabolic affine Grassmannian
and character sheaves}, Selecta Math.\ (N.S.) {\bf 14} (2009), no.~3--4, 607--628.

\bibitem[Ga]{garfinkle}
D.~Garfinkle, {\it The annihilators of irreducible Harish-Chandra modules for $\mathrm{SU}(p,q)$ and other type $A_{n-1}$ groups},
Amer.\ J.\ Math.\ {\bf 115} (1993), no.~2, 305–-369. 

\bibitem[Gi]{ginzburg}
V.~Ginzburg, {\it Admissible modules on a symmetric space},
Ast\'erisque {\bf 173--174} (1989), 199--255.

\bibitem[Gr]{grojnowski}
I.~Grojnowski, {\it Character sheaves on symmetric spaces},
    PhD thesis, Massachusetts Institute of Technology, 1992, available at \texttt{www.dpmms.cam.ac.uk/$\sim$groj/thesis.ps} .

\bibitem[GPf]{gpf}
M.~Geck and G.~Pfeiffer, {\it Characters of finite Coxeter groups and
Iwahori--Hecke algebras}, London Mathematical Society Monographs, New
  Series, vol.~21, Oxford University Press, New York, 2000.

\bibitem[K1]{kato:exotic}
S.~Kato, {\it An exotic Deligne-Langlands correspondence for symplectic groups},
Duke Math.\ J.\ {\bf 148} (2009), no.~2, 305--371.

\bibitem[K2]{kato:deformations}
S.~Kato, {\it Deformations of nilpotent cones and Springer correspondences},
Amer.\ J.\ Math.\ {\bf 133} (2011), no.~2, 519--553.

\bibitem[KL]{kl}
D.~Kazhdan and G.~Lusztig, {\it Representations of Coxeter groups and Hecke algebras},
Invent.\ Math.\ \textbf{53} (1979), 165--184.

\bibitem[L]{lusztig:charsheaves4}
G.~Lusztig, {\it Character sheaves IV},
Adv.\ Math.\ {\bf 59} (1986), 1--63.

\bibitem[LV]{lv}
G.~Lusztig and D.~A.~Vogan, jr., {\it Singularities of closures of $K$-orbits on
flag manifolds}, Invent.\ Math.\ {\bf 71} (1983), 365--379.


\bibitem[Mag]{magyar}
P.~Magyar, {\it Bruhat order for two flags and a line},
J.\ Algebraic Combin.\ {\bf 21} (2005), 71--101.

\bibitem[Mat]{matsuki}
T.~Matsuki, {\it An example of orthogonal triple flag variety of finite type}, arXiv:1011.6468v1.

\bibitem[McG]{mcgovern}
W.~M.~McGovern, {\it On the Spaltenstein-Steinberg map for classical Lie algebras},
Comm.\ Algebra {\bf 27} (1999), no.~6, 2979–-2993. 

\bibitem[RS]{rs}
R.~W.~Richardson and T.~A.~Springer, {\it The Bruhat order on symmetric varieties},
Geom.\ Dedicata {\bf 35} (1990), no.~1--3, 389--436.

\bibitem[Sp]{spaltenstein}
N.~Spaltenstein, {\it The fixed point set of a unipotent transformation on the 
flag manifold}, Nederl.\ Akad.\ Wetensch.\ Proc.\ Ser.\ A {\bf 79} = Indag.\ Math.\  
{\bf 38}  (1976), no.~5, 452--456.

\bibitem[St1]{st1}
R.~Steinberg, {\it On the desingularization of the unipotent variety},
Invent.\ Math.\ {\bf 36} (1976), 209--224.

\bibitem[St2]{st2}
R.~Steinberg, {\it An occurrence of the Robinson--Schensted correspondence},
J.\ Algebra \textbf{113} (1988), 523--528.

\bibitem[Trap]{trapa}
P.~E.~Trapa, {\it Generalized Robinson--Schensted algorithms for real groups},
Internat.\ Math.\ Res.\ Notices {\bf 1999}, no.~15, 803–-834. 

\bibitem[Trav]{travkin}
R.~Travkin, {\it Mirabolic Robinson--Schensted--Knuth correspondence},
Selecta Math.\ (N.S.)  {\bf 14} (2009), no.~3--4, 727--758.

\end{thebibliography}
\end{document}